\documentclass[a4paper,12pt,reqno]{amsart}
\usepackage{amsfonts, color}
\usepackage{amsmath}
\usepackage{amssymb}
\usepackage[a4paper]{geometry}
\usepackage{mathrsfs}
\usepackage[colorlinks]{hyperref}
\usepackage{hyperref}
\renewcommand\eqref[1]{(\ref{#1})}

\makeatletter
\newcommand*{\mint}[1]{%
  \mint@l{#1}{}%
}
\newcommand*{\mint@l}[2]{%
  \@ifnextchar\limits{%
    \mint@l{#1}%
  }{%
    \@ifnextchar\nolimits{%
      \mint@l{#1}%
    }{%
      \@ifnextchar\displaylimits{%
        \mint@l{#1}%
      }{%
        \mint@s{#2}{#1}%
      }%
    }%
  }%
}
\newcommand*{\mint@s}[2]{%
  \@ifnextchar_{%
    \mint@sub{#1}{#2}%
  }{%
    \@ifnextchar^{%
      \mint@sup{#1}{#2}%
    }{%
      \mint@{#1}{#2}{}{}%
    }%
  }%
}
\def\mint@sub#1#2_#3{%
  \@ifnextchar^{%
    \mint@sub@sup{#1}{#2}{#3}%
  }{%
    \mint@{#1}{#2}{#3}{}%
  }%
}
\def\mint@sup#1#2^#3{%
  \@ifnextchar_{%
    \mint@sup@sub{#1}{#2}{#3}%
  }{%
    \mint@{#1}{#2}{}{#3}%
  }%
}
\def\mint@sub@sup#1#2#3^#4{%
  \mint@{#1}{#2}{#3}{#4}%
}
\def\mint@sup@sub#1#2#3_#4{%
  \mint@{#1}{#2}{#4}{#3}%
}
\newcommand*{\mint@}[4]{%
  \mathop{}%
  \mkern-\thinmuskip
  \mathchoice{%
    \mint@@{#1}{#2}{#3}{#4}%
        \displaystyle\textstyle\scriptstyle
  }{%
    \mint@@{#1}{#2}{#3}{#4}%
        \textstyle\scriptstyle\scriptstyle
  }{%
    \mint@@{#1}{#2}{#3}{#4}%
        \scriptstyle\scriptscriptstyle\scriptscriptstyle
  }{%
    \mint@@{#1}{#2}{#3}{#4}%
        \scriptscriptstyle\scriptscriptstyle\scriptscriptstyle
  }%
  \mkern-\thinmuskip
  \int#1%
  \ifx\\#3\\\else_{#3}\fi
  \ifx\\#4\\\else^{#4}\fi
}
\newcommand*{\mint@@}[7]{%
  \begingroup
    \sbox0{$#5\int\m@th$}%
    \sbox2{$#5\int_{}\m@th$}%
    \dimen2=\wd0 %
    \let\mint@limits=#1\relax
    \ifx\mint@limits\relax
      \sbox4{$#5\int_{\kern1sp}^{\kern1sp}\m@th$}%
      \ifdim\wd4>\wd2 %
        \let\mint@limits=\nolimits
      \else
        \let\mint@limits=\limits
      \fi
    \fi
    \ifx\mint@limits\displaylimits
      \ifx#5\displaystyle
        \let\mint@limits=\limits
      \fi
    \fi
    \ifx\mint@limits\limits
      \sbox0{$#7#3\m@th$}%
      \sbox2{$#7#4\m@th$}%
      \ifdim\wd0>\dimen2 %
        \dimen2=\wd0 %
      \fi
      \ifdim\wd2>\dimen2 %
        \dimen2=\wd2 %
      \fi
    \fi
    \rlap{%
      $#5%
        \vcenter{%
          \hbox to\dimen2{%
            \hss
            $#6{#2}\m@th$%
            \hss
          }%
        }%
      $%
    }%
  \endgroup
}
\numberwithin{equation}{section}
\theoremstyle{plain}
\newtheorem{thm}{Theorem}[section]

\newtheorem{cor}[thm]{Corollary}

\theoremstyle{definition}
\newtheorem{defn}[thm]{Definition}
\newtheorem{rem}[thm]{Remark}

\title[Schrödinger equation for Sturm-Liouville operator with singular  ]{Schrödinger equation for Sturm-Liouville operator with singular propagation and potential}
\author[M. Ruzhansky]{Michael Ruzhansky}
\address{
  Michael Ruzhansky:
  \endgraf
  Department of Mathematics: Analysis, Logic and Discrete Mathematics
  \endgraf
  Ghent University, Belgium
  \endgraf
 and
  \endgraf
  School of Mathematical Sciences
  \endgraf
  Queen Mary University of London
  \endgraf
  United Kingdom
  \endgraf
  {\it E-mail address} {\rm michael.ruzhansky@ugent.be}
  }

\author[A. Yeskermessuly]{Alibek Yeskermessuly}
\address{
  Alibek Yeskermessuly:
    \endgraf
  Department of Mathematics: Analysis, Logic and Discrete Mathematics
  \endgraf
  Ghent University, Belgium
  \endgraf
 and 
   \endgraf
  Altynsarin Arkalyk Pedagogical Institute,  \\ 
  \endgraf
  Arkalyk, Kazakhstan
  \endgraf
  {\it E-mail address} {\rm alibek.yeskermessuly@gmail.com}}

\begin{document}

\thanks{The authors are supported by the FWO Odysseus 1 grant G.0H94.18N: Analysis and Partial Differential Equations and by the Methusalem programme of the Ghent University Special Research Fund (BOF) (Grant number 01M01021). Michael Ruzhansky is also supported by EPSRC grants EP/R003025/2 and EP/V005529/1.
}

%
%

%
%

%
\maketitle              

\begin{abstract}
In this paper we consider an initial/boundary value problem for the Schrödinger equation with the Hamiltonian involving the fractional Sturm-Liouville operator with singular propagation and potential. To construct a solution, first considering the coefficients in a regular sense, the method of separation of variables is used, which leads the solution of the equation to the eigenvalue and eigenfunction problem of the Sturm-Liouville operator. Next, using the Fourier series expansion in eigenfunctions, a solution to the Schrödinger equation is constructed. Important estimates related to the Sobolev space are also obtained.
In addition, the equation is studied in the case where the initial data, propagation and potential are strongly singular. For this case, the concept of\, ``very weak solutions'' is used. The existence, uniqueness, negligibility and consistency of very weak solution of the Schrödinger equation are established.

\end{abstract}

\section{Introduction}

The main goal of this paper is to establish the existence of physical solutions for the Schrödinger equation, specifically when it involves the Sturm-Liouville operator with singular potentials. When tackling problems with strong singularities, a prior study by \cite{Gar-Ruz} introduced the concept of \textquotedblleft very weak solutions\textquotedblright. This approach is necessary because when the equation involves products of various terms, it can no longer be clearly defined in spaces of distributions. Consequently, we require an alternative way to determine the well-posedness of the equation.

The development of very weak solutions for various types of problems continued in several works, such as \cite{ARST1}, \cite{ARST2}, \cite{ARST3}, \cite{CRT1}, \cite{CRT2}, \cite{Lan-Ham}, \cite{Ruz-Tok}, \cite{R-Y}. In the works \cite{RYS} and \cite{RYes} the concept of very weak solutions of the wave equation for the Sturm-Liouville operator with singular potentials in bounded domains was expanded.

It is known that the Schrödinger equation can be simplified into ordinary linear equations using the \textquotedblleft separation of variables\textquotedblright \,method, see e.g. \cite{Separ}. To present our main findings, we provide some initial information about the Sturm-Liouville operator with singular potentials. Savchuk and Shkalikov's study in \cite{Sav-Shk} yielded eigenvalues and eigenfunctions for this operator. Additionally, studies in \cite{N-zSk}, \cite{Savch}, \cite{Sav-Shk2}, and \cite{SV} explored the Sturm-Liouville operator with potential-distributions. To establish the framework for very weak solutions, our focus is primarily on estimating solutions for more regular problems while also considering the impact of a regularization parameter on these solutions.

For further reasoning and obtaining our results, we need some preliminaries about the Sturm-Liouville operator with singular potentials. More specifically, we consider the Sturm-Liouville operator $\mathcal{L}$ generated on the interval (0,1) by the differential expression 
\begin{equation}\label{St-L}
    \mathcal{L}y:=-\frac{d^2}{dx^2}y+q(x)y,
\end{equation}
with the boundary conditions
\begin{equation}\label{Dirihle}
    y(0)=y(1)=0. 
\end{equation}
The potential $q$ is defined as
\begin{equation}\label{con-q}
    q(x)=\nu'(x)\geq 0, \qquad \nu\in L^2(0,1).
\end{equation}

The eigenvalues of the Sturm-Liouville operator $\mathcal{L}$ generated on the interval (0,1) by the differential expression \eqref{St-L} with the boundary conditions \eqref{Dirihle} are real (\cite{Kost}) and given by
\begin{equation}\label{e-val}
    \lambda_n=(\pi n)^2(1+o(n^{-1})),\qquad n=1,2,...,
\end{equation}
and the corresponding eigenfunctions are
\begin{equation}\label{sol-SL}
    \Tilde{\phi}_n(x)=r_n(x)\sin\theta_n(x),
\end{equation}
where
$$r_n(x)=\exp{\left(-\int\limits_0^x \nu(s)\cos2\theta_n(s)ds+o(1)\right)}=1+o(1),$$
$$\theta_n(x)=\sqrt{\lambda_n}x+o(1),$$
for $n\to \infty$.
According to \eqref{con-q}, \eqref{e-val} and \eqref{sol-SL} it is clear that the $\Tilde{\varphi}_n$ are real. Here and below we will have the positive operator $\langle \mathcal{L}y, y\rangle\geq 0$, which implies that all eigenvalues $\lambda_n$ are real and non-negative.

The first derivatives of $\Tilde{\phi_n}$ are given by the formulas
\begin{equation}\label{phi-der}
    \Tilde{\phi}'_n(x)=\sqrt{\lambda_n}r_n(x)\cos(\theta_n(x))+\nu(x)\Tilde{\phi}_n(x).
\end{equation}

According to Theorem 2 in \cite{Savch} we have
\begin{equation}\label{phi-sav}
\Tilde{\phi}_n(x)=\sin{\sqrt{\lambda_n}x}+\psi_n(x), \quad n=1,2,...,\quad \sum\limits_{n=1}^\infty \|\psi_n\|^2\leq C \int\limits_0^1|\nu(x)|^2dx.
\end{equation}

On the other hand, we can estimate the $\|\Tilde{\phi}_n\|_{L^2}$ using the formula \eqref{sol-SL} as follows
\begin{eqnarray}\label{est-high}
\|\Tilde{\phi}_n\|^2_{L^2}&\lesssim& \exp{\left(\|\nu\|_{L^2}+\lambda^{-\frac{1}{2}}\|\nu\|^2_{L^2}\right)}<\infty.
\end{eqnarray}

Also, according to Theorem 4 in \cite{Sav-Shk}, we have
\begin{equation}\label{est_low}
  \Tilde{\phi}_n(x)=\sin(\pi nx)+o(1)  
\end{equation}
for sufficiently large $n$. Along with \eqref{sol-SL}, we see that there exist some $C_0>0$, such that
\begin{equation}\label{low-est}
0<C_0\leq\|\Tilde{\phi}_n\|_{L^2}<\infty
\end{equation}
for all $n$. 

Since the eigenfunctions of the Sturm-Liouville operator form an orthogonal basis in $L^2(0,1)$, we normalize them for further use  
\begin{equation}\label{norm-phi}
  \phi_n(x)=\frac{\Tilde{\phi}_n(x)}{\sqrt{\langle \Tilde{\phi}_n,\Tilde{\phi}_n}\rangle}=\frac{\Tilde{\phi}_n(x)}{\|\Tilde{\phi}_n\|_{L^2}}.  
\end{equation}

\section{Non-homogeneous Schrödinger equation}

We consider the non-homogeneous Schrödinger equation with initial/boundary conditions
\begin{equation}\label{nonh}
    \left\{\begin{array}{l}
    i\partial_t u(t,x)+a(t)\mathcal{L}^s u(t,x)=f(t,x),\qquad (t,x)\in [0,T]\times (0,1),\\
    u(0,x)=u_0(x),\quad x\in (0,1),\\
    u(t,0)=0=u(t,1),\quad t\in[0,T],
    \end{array}\right.
\end{equation}
where $a(t)\geq a_0>0$ for $t\in [0,T]$ and $a\in L^\infty [0,T]$, $s\in\mathbb{R}$, with operator $\mathcal{L}$ defined by 
\begin{equation}\label{op-L}
  \mathcal{L}=-\frac{\partial^2}{\partial x^2}+q(x),\qquad x\in(0,1),  
\end{equation}
and $q=\nu'\geq0$, $\nu \in L^2(0,1)$.

It is well known (\cite{RYS}, \cite{RYes}) that the general solution to this equation is
$$u(t,x)=u_{1}(t,x)+u_{2}(t,x),$$
where $u_{1}(t,x)$ is the general solution of the homogeneous Schrödinger equation
\begin{equation}\label{C.p1}
         i\partial_t u(t,x)+a(t)\mathcal{L}^s u(t,x)=0,\qquad (t,x)\in [0,T]\times (0,1),
\end{equation}
with initial condition
\begin{equation}\label{C.p2} 
u(0,x)=u_0(x),\,\,\, x\in (0,1),
\end{equation}
and with Dirichlet boundary conditions
\begin{equation}\label{C.p3}
u(t,0)=0=u(t,1),\qquad t\in [0,T],
\end{equation}
and $u_{2}(t,x)$ is the particular solution of the non-homogeneous Schrödinger equation with initial/boundary conditions \eqref{nonh}. In other words, to get a solution to \eqref{nonh} we need to consider problem \eqref{C.p1}-\eqref{C.p3}.

In our results below, concerning the initial/boundary problem \eqref{C.p1}-\eqref{C.p3}, as the preliminary step we first carry out the analysis in the regular case for bounded $q \in L^\infty(0,1)$. In this case, we obtain the well-posedness in the Sobolev spaces $W^k_\mathcal{L}$ associated to the operator $\mathcal{L}$: we define the Sobolev spaces $W^k_\mathcal{L}$ associated to $\mathcal{L}$, for any $k \in \mathbb{R}$, as the space  
$$W^k_\mathcal{L}:=\left\{f\in \mathcal{D}'_\mathcal{L}(0,1):\,\mathcal{L}^{k/2}f\in L^2(0,1)\right\},$$
with the norm $\|f\|_{W^k_\mathcal{L}}:=\|\mathcal{L}^{k/2}f\|_{L^2}$. The global space of distributions $\mathcal{D}'_\mathcal{L}(0,1)$ is defined as follows.

The space $C^\infty_\mathcal{L}(0,1):=\mathrm{Dom}(\mathcal{L}^\infty)$ is called the space of test functions for $\mathcal{L}$, where we define 
$$\mathrm{Dom}(\mathcal{L}^\infty):=\bigcap\limits_{m=1}^\infty \mathrm{Dom}(\mathcal{L}^m),$$
where $\mathrm{Dom}(\mathcal{L}^m)$ is the domain of the operator $\mathcal{L}^m$, in turn defined as
$$\mathrm{Dom}(\mathcal{L}^m):=\left\{f\in L^2(0,1): \mathcal{L}^j f\in \mathrm{Dom}(\mathcal{L}),\,\, j=0,1,2,...,m-1\right\}.$$
The Fréchet topology of $C^\infty_\mathcal{L}(0,1)$ is given by the family of norms 
\begin{equation}\label{frechet}
    \|\phi\|_{C^m_\mathcal{L}}:=\max\limits_{j\leq m}\|\mathcal{L}^j\phi\|_{L^2(0,1)},\quad m\in \mathbb{N}_0,\,\, \phi\in C^\infty_\mathcal{L}(0,1).
\end{equation}
The space of $\mathcal{L}$-distributions
$$\mathcal{D}'_\mathcal{L}(0,1):=\mathbf{L}\left(C^\infty_\mathcal{L}(0,1),\mathbb{C}\right)$$
is the space of all linear continuous functionals on $C^\infty_\mathcal{L}(0,1)$. For $\omega \in \mathcal{D}'_\mathcal{L}(0,1)$ and $\phi\in C^\infty_\mathcal{L}(0,1)$, we shall write 
$$\omega(\phi)=\langle \omega, \phi\rangle.$$
For any $\psi \in C^\infty_\mathcal{L}(0,1)$, the functional 
$$C^\infty_\mathcal{L}(0,1)\ni \phi \mapsto \int\limits_0^1 \psi(x)\phi(x)dx$$
is an $\mathcal{L}$-distribution, which gives an embedding $\psi \in C^\infty_\mathcal{L}(0,1)\hookrightarrow \mathcal{D}'_\mathcal{L}(0,1)$.

We introduce the spaces $C^j([0,T],W^k_\mathcal{L}(0,1))$ given by the family of norms
\begin{equation}
    \|f\|_{C^n([0,T],W^k_\mathcal{L}(0,1))}=\max\limits_{0\leq t\leq T}\sum\limits_{j=0}^n\left\|\partial^j_t f(t,\cdot)\right\|_{W^k_\mathcal{L}},
\end{equation}
where $k\in \mathbb{R}, \, f\in C^j([0,T],W^k_\mathcal{L}(0,1)).$
 
\begin{thm}\label{th1}
Assume that $q \in L^\infty(0,1)$, $q\geq0$, $a(t)\geq a_0>0$ for all $t\in [0,T],$ and $a\in L^\infty[0,T]$. For any $k\in \mathbb{R}$, if the initial condition satisfies $u_0 \in W^k_\mathcal{L}$ then the Schrödinger equation \eqref{C.p1} with the initial/boundary conditions  \eqref{C.p2}-\eqref{C.p3}  has a unique solution $u\in C([0,T], W^{k}_\mathcal{L})$. We also have the following estimates:
\begin{equation}\label{eq2.1}
    \|u(t,\cdot)\|_{L^2}\lesssim \|u_0\|_{L^2},
\end{equation}
\begin{equation}\label{eq2.2}
    \|\partial_{t}u(t,\cdot)\|_{L^2}\lesssim\|a\|_{L^\infty[0,T]}\|u_0\|_{W^{2s}_\mathcal{L}}.
\end{equation}
When $s=1$, we also have
\begin{equation}\label{est3}
    \|\partial_{x}u(t,\cdot)\|_{L^2}\lesssim\|u_0\|_{W^1_\mathcal{L}}\left(1+\|\nu\|_{L^2}\right)+\|u_0\|_{L^2}\|\nu\|_{L^\infty},
\end{equation}
\begin{equation}\label{est2.4}
\left\|\partial^2_xu(t,\cdot)\right\|_{L^2} 
\lesssim \|q\|_{L^\infty} \|u_0\|_{L^2}+\|u_0\|_{W^2_\mathcal{L}},
\end{equation}
\begin{equation}\label{est5}
 \left\|u(t, \cdot)\right\|_{W^k_\mathcal{L}}\lesssim\left\|u_0\right\|_{W^k_\mathcal{L}},
\end{equation}

where the constants in these inequalities are independent of $u_0$, $\nu$, $q$ and $a$.
\end{thm} 

We note that $q\in L^\infty(0,1)$ implies that $\nu \in L^\infty(0,1)$ and hence $\nu\in L^2(0,1)$, so that the formulas in the introduction hold true.

\begin{proof}
We apply the technique of the separation of variables (see, e.g. \cite{Separ}). In particular, we are looking for a solution of the form
$$u(t,x)=T(t)X(x),$$
where $T(t)$, $X(x)$ are unknown functions that must be determined. Substituting $u(t,x)=T(t)X(x)$ into equation \eqref{C.p1} and after simple transformations, we get for the function $T(t)$ the equation
\begin{equation}\label{3}
        T'(t)=i\mu a(t)T(t),  \qquad t\in [0,T],
\end{equation}
and for the function $X(x)$ we get
\begin{equation}\label{4}
    \mathcal{L}^sX(x)=\mu X(x),
\end{equation}
where $\mu$ is a spectral parameter. When $s=1$, we obtain the Sturm-Liouville boundary value problem
\begin{equation}\label{4+}
 \mathcal{L}X(x):=-X''(x)+q(x)X(x)=\lambda X(x),   
\end{equation}
\begin{equation}\label{5}
    X(0)=X(1)=0.
\end{equation}

The equation \eqref{4+} with the boundary condition \eqref{5} has the eigenvalues of the form \eqref{e-val} with the corresponding eigenfunctions of the form \eqref{sol-SL} of the Sturm-Liouville operator $\mathcal{L}$ generated by the differential expression \eqref{St-L}. Substituting 
$$\mu_n=\lambda_n^s,$$
we get the eigenvalues of the form \eqref{e-val} and the corresponding eigenfunctions of the form \eqref{sol-SL} for the equation \eqref{4}, i.e. 
\begin{equation}\label{l^s}
  \mathcal{L}^s\phi_n(x)=\lambda_n^s\phi_n(x).  
\end{equation}

The solution of the equation \eqref{3} with the initial conditions \eqref{C.p2} is
$$T_n(t)=D_n e^{i\lambda_n^s\int\limits_0^t a(\tau)d\tau},$$
where
$$D_n=\int\limits_0^1u_0(x)\phi_n(x)dx.$$

Thus, the solution of the homogeneous Schrödinger equation \eqref{C.p1}  with the initial/boundary conditions \eqref{C.p2}-\eqref{C.p3} has the form
\begin{equation}\label{23}
    u(t,x)=\sum\limits_{n=1}^\infty D_ne^{i \lambda_n^s \int\limits_0^t a(\tau)d\tau}\phi_n(x). 
\end{equation}

Further we will prove that $u\in C^2([0,T],L^2(0,1))$. By using the Cauchy-Schwarz inequality and fixed $t$, we can deduce that

\begin{eqnarray}\label{25-0}
\|u(t, \cdot)\|^2_{L^2}&=&\int\limits_0^1|u(t,x)|^2dx =\int\limits_0^1\left|\sum\limits_{n=1}^\infty D_n e^{i \lambda_n^s \int\limits_0^t a(\tau)d\tau}\phi_n(x)\right|^2dx\nonumber\\
&\lesssim& \int\limits_0^1\sum\limits_{n=1}^\infty\left|D_n e^{i \lambda_n^s \int\limits_0^t a(\tau)d\tau}\right|^2|\phi_n(x)|^2dx.
\end{eqnarray}
According to \eqref{e-val}, \eqref{norm-phi} using Euler's formula and Parseval's identity, we obtain
\begin{eqnarray}\label{25}
\|u(t, \cdot)\|^2_{L^2}&\lesssim& \int\limits_0^1\sum\limits_{n=1}^\infty\left|D_n e^{i \lambda_n^s \int\limits_0^t a(\tau)d\tau}\right|^2|\phi_n(x)|^2dx=\sum\limits_{n=1}^\infty\left|D_n\right|^2\int\limits_0^1|\phi_n(x)|^2dx\nonumber\\
&=&\sum\limits_{n=1}^\infty\left|D_n\right|^2=\int\limits_0^1|u_0(x)|^2dx=\|u_0\|^2_{L^2}.
\end{eqnarray}

Since $a\in L^\infty[0,T]$ and using \eqref{25} we obtain
\begin{eqnarray}\label{t26}
\|\partial_t u(t,\cdot)\|^2_{L^2}&=&\int\limits_0^1|\partial_tu(t,x)|^2dt = \int\limits_0^1\left|\sum\limits_{n=1}^\infty\left((i \lambda_n^s) a(t)D_n e^{i \lambda_n^s\int\limits_0^t a(\tau)d\tau}\phi_n(x)\right)\right|^2dx \nonumber\\ 
&\lesssim& \int\limits_0^1\sum\limits_{n=1}^\infty|a(t)|^2\left|\lambda_n^s D_n e^{i \lambda_n^s \int\limits_0^t a(\tau)d\tau}\right|^2|\phi_n(x)|^2dx \nonumber\\
&\leq&\sum\limits_{n=1}^\infty\|a\|^2_{L^\infty[0,T]}\left|\lambda_n^s D_n \right|^2\int\limits_0^1|\phi_n(x)|^2dx\nonumber\\
&=&\|a\|^2_{L^\infty[0,T]}\sum\limits_{n=1}^\infty|\lambda_n^s D_n |^2.
\end{eqnarray}
Since $\lambda_n$ are eigenvalues and $\phi_n$ are eigenfunctions of the operator $\mathcal{L},$ using Parseval's identity we obtain 
\begin{eqnarray}\label{21-1}\sum\limits_{n=1}^\infty|\lambda_n^s D_n|^2&=& \sum\limits_{n=1}^\infty\left|\lambda_n^s\int\limits_0^1 u_0(x)\phi_n(x)dx\right|^2 =\sum\limits_{n=1}^\infty\left| \int\limits_0^1 \lambda_n^su_0(x)\phi_n(x)dx\right|^2\nonumber\\
&=&\sum\limits_{n=1}^\infty\left| \int\limits_0^1 \mathcal{L}^su_0(x)\phi_n(x)dx\right|^2=\|\mathcal{L}^su_0\|^2_{L^2}=\|u_0\|^2_{W^{2s}_{\mathcal{L}}}.
\end{eqnarray}
Thus, 
$$\|\partial_t u(t,\cdot)\|^2_{L^2}\lesssim \|a\|^2_{L^\infty[0,T]}\|u_0\|^2_{W^{2s}_\mathcal{L}}.$$

Let $s=1$, then to estimate the norm of $\partial_x u(t,\cdot)$ in $L^2$ we use \eqref{phi-der} and \eqref{norm-phi} for $\phi'_n$:
\begin{eqnarray*}
\|\partial_x u(t,\cdot)\|^2_{L^2}&=&\int\limits_0^1|\partial_xu(t,x)|^2dt=
\int\limits_0^1\left|\sum\limits_{n=1}^\infty D_ne^{i \lambda_n \int\limits_0^t a(\tau)d\tau}\phi'_n(x)\right|^2dx \nonumber \\ &=&\int\limits_0^1\left|\sum\limits_{n=1}^\infty D_ne^{i \lambda_n \int\limits_0^t a(\tau)d\tau}\left(\frac{\sqrt{\lambda_n}r_n(x) \cos\theta_n(x)}{\|\Tilde{\phi}_n\|_{L^2}}+\nu(x)\phi_n(x)\right)\right|^2dx.
\end{eqnarray*}
According to \eqref{25}, \eqref{est-high} and \eqref{est_low}, there exist some $C_0>0$, such that $C_0<\|\Tilde{\phi}_n\|_{L^2}<\infty$, so that
\begin{eqnarray*}
\|\partial_x u(t,\cdot)\|^2_{L^2}&\lesssim&\sum\limits_{n=1}^\infty\left(\left|\sqrt{\lambda_n}D_n\right|^2\int\limits_0^1|r_n(x)|^2dx\right)+\sum\limits_{n=1}^\infty\left(|D_n|^2\int\limits_0^1|\nu(x)\phi_n(x)|^2dx\right).
\end{eqnarray*}
Here for $r_n(x)$ according Theorem 2 in \cite{Savch}, we have
$$r_n(x)=1+\rho_n(x),\quad \|\rho_n\|_{L^2}\lesssim \|\nu\|_{L^2},$$
where the constant is independent of $\nu$ and $n$. Therefore,
$$\int\limits_0^1|r_n(x)|^2dx\lesssim 1+\|\nu\|^2_{L^2}.$$
For second term we obtain
$$\int\limits_0^1|\nu(x)\phi_n(x)|^2dx\leq \|\nu\|^2_{L^\infty}\|\phi_n\|^2_{L^2}=\|\nu\|^2_{L^\infty},$$
since $\{\phi_n\}$ is an orthonormal basis in $L^2$.
Using the property of the operator $\mathcal{L}$ and the Parseval identity, we obtain
\begin{eqnarray*}
 \sum\limits_{n=1}^\infty\left|\sqrt{\lambda_n}D_n\right|^2&=& \sum\limits_{n=1}^\infty\left|\int\limits_0^1\sqrt{\lambda_n}u_0(x)\phi_n(x)dx\right|^2\\
 &=&\sum\limits_{n=1}^\infty\left|\int\limits_0^1\mathcal{L}^\frac{1}{2}u_0(x)\phi_n(x)dx\right|^2=\left\|\mathcal{L}^\frac{1}{2}u_0\right\|^2_{L^2}=\|u_0\|^2_{W^1_\mathcal{L}}.
\end{eqnarray*}
Using the last relations, we obtain
\begin{eqnarray}\label{u_x}
\|\partial_x u(t,\cdot)\|^2_{L^2}&\lesssim&\sum\limits_{n=1}^\infty\left|\sqrt{\lambda_n}D_n\right|^2\left(1+\|\nu\|^2_{L^2}\right)+\sum\limits_{n=1}^\infty \left|D_n\right|^2\|\nu\|^2_{L^\infty}\nonumber\\
&\leq& \|u_0\|^2_{W^1_\mathcal{L}}\left(1+\|\nu\|^2_{L^2}\right)+\|u_0\|^2_{L^2}\|\nu\|^2_{L^\infty},
\end{eqnarray}
implying \eqref{est3}.

Let us get next estimate by using that $\phi''_n(x)=(q(x)- \lambda_n)\phi_n(x)$ in the case when $s=1$,
\begin{eqnarray}\label{u_xx}
\left\|\partial_x^2u(t, \cdot)\right\|^2_{L^2}&=&\int\limits_0^1\left|\partial^2_xu(t,x)\right|^2dx=\int\limits_0^1\left|\sum\limits_{n=1}^\infty D_n e^{i \lambda_n \int\limits_0^t a(\tau)d\tau} \phi''_n(x)\right|^2dx\nonumber\\
&\lesssim& \int\limits_0^1\left(\sum\limits_{n=1}^\infty \left|{D_n} e^{i \lambda_n \int\limits_0^t a(\tau)d\tau}\right|^2|(q(x)- \lambda_n)\phi_n(x)|^2\right)dx\nonumber\\
&\lesssim&\int\limits_0^1|q(x)|^2\sum\limits_{n=1}^\infty \left|D_n\right|^2|\phi_n(x)|^2dx+\int\limits_0^1\sum\limits_{n=1}^\infty |\lambda_n D_n|^2|\phi_n(x)|^2 dx\nonumber\\
&\leq&\|q\|^2_{L^\infty}\sum\limits_{n=1}^\infty |D_n|^2 +\sum\limits_{n=1}^\infty \left|\lambda_n D_n\right|^2. 
\end{eqnarray}
Taking into account \eqref{21-1} for the last terms in \eqref{u_xx}, we obtain
\begin{eqnarray}\label{lB}
 \sum\limits_{n=1}^\infty\left|\lambda_n D_n\right|^2=\|u_0\|^2_{W^2_\mathcal{L}},
\end{eqnarray}
using the last expressions and \eqref{21-1} we finally get
\begin{eqnarray*}
\left\|\partial^2_xu(t,\cdot)\right\|^2_{L^2} 
&\lesssim &\|q\|^2_{L^\infty} \|u_0\|^2_{L^2}+\|u_0\|^2_{W^2_\mathcal{L}},
\end{eqnarray*}
implying \eqref{est2.4}.

Let us carry out the last estimate \eqref{est5} using that $\mathcal{L}^ku=\lambda_n^ku$ and Parseval's identity:
\begin{eqnarray*}
 \left\|u(t, \cdot)\right\|^2_{W^k_\mathcal{L}}&=&\left\|\mathcal{L}^\frac{k}{2}u(t, \cdot)\right\|^2_{L^2}=\int\limits_0^1\left|\mathcal{L}^\frac{k}{2}u(t,x)\right|^2dx=\int\limits_0^1\left|\sum\limits_{n=1}^\infty  D_n e^{i \lambda_nt}\lambda_n^\frac{k}{2}\phi_n(x)\right|^2dx\\
 &\lesssim&\sum\limits_{n=1}^\infty \left| \lambda_n^\frac{k}{2}D_n\right|^2=\left\|\mathcal{L}^\frac{k}{2}u_0\right\|^2_{L^2}=\left\|u_0\right\|^2_{W^k_\mathcal{L}}.
\end{eqnarray*}
The proof of Theorem \ref{th1} is complete.
\end{proof}

The following statement removes the reliance on Sobolev spaces with respect to $\mathcal{L}$ while sacrificing the regularity of the data. This statement will be important for further analysis.

\begin{cor}\label{cor1} Let $s=1$.
Assume that $q,\, \nu \in L^\infty(0,1)$, $q\geq0$, $a(t)\geq a_0>0$ for all $t\in [0,T]$, and $a\in L^\infty[0,T]$. If the initial condition satisfies $u_0 \in L^2(0,1)$ and $u_0''\in L^2(0,1)$ then the Schrödinger equation \eqref{C.p1} with the initial/boundary conditions  \eqref{C.p2}-\eqref{C.p3}  has a unique solution $u\in C([0,T], L^2(0,1))$ which satisfies the estimates
\begin{equation}\label{ec1}
      \|u(t,\cdot)\|_{L^2}\lesssim \|u_0\|_{L^2},
\end{equation}
\begin{equation}\label{ec2}
\|\partial_t u(t,\cdot)\|_{L^2}\lesssim \|a\|_{L^\infty[0,T]}\left(\|u''_0\|_{L^2}+\|q\|_{L^\infty}\|u_0\|_{L^2}\right),
\end{equation}
\begin{eqnarray}\label{ec3}
\|\partial_{x}u(t,\cdot)\|_{L^2}\lesssim \left(\|u''_0\|_{L^2}+\|q\|_{L^\infty}\|u_0\|_{L^2}\right)\left(1 +\|\nu\|_{L^2}\right)+  \|u_0\|_{L^2}\|\nu\|_{L^\infty},
\end{eqnarray}
\begin{equation}\label{ec4}
\left\|\partial^2_xu(t,\cdot)\right\|_{L^2} 
\lesssim \|u''_0\|_{L^2}+\|q\|_{L^\infty}\|u_0\|_{L^2},
\end{equation}
where  the constants in these inequalities are independent of $u_0$, $q$ and $a$.
\end{cor}
\begin{proof}
The inequality \eqref{ec1} immediately follows from \eqref{eq2.1}. Let us move on to estimating the inequality \eqref{ec2}.
In Theorem \ref{th1} we obtained estimates with respect to the operator $\mathcal{L}$, but here we want to obtain estimates with respect to the initial condition $u_0$ and potential $q(x)$. 

By \eqref{t26} we have
\begin{equation*}
\|\partial_t u(t,\cdot)\|^2\lesssim \|a\|^2_{L^\infty[0,T]}\sum\limits_{n=1}^\infty|\lambda_n D_n |^2.
\end{equation*}
Since $\lambda_n$ are the eigenvalues of the operator $\mathcal{L}$, we obtain 
\begin{eqnarray}\label{21}\sum\limits_{n=1}^\infty|\lambda_n D_n|^2&=& \sum\limits_{n=1}^\infty\left| \int\limits_0^1 \lambda_n u_0(x)\phi_n(x)dx\right|^2
=\sum\limits_{n=1}^\infty\left| \int\limits_0^1 \left(-u''_0(x)+q(x)u_0(x)\right)\phi_n(x)dx\right|^2\nonumber \\
&\lesssim& \sum\limits_{n=1}^\infty\left| \int\limits_0^1 u''_0(x)\phi_n(x)dx\right|^2+\sum\limits_{n=1}^\infty\left|\int\limits_0^1q(x)u_0(x)\phi_n(x)dx\right|^2.
      \end{eqnarray}
Using Parseval's identity for the first and second term in \eqref{21} and since $q\in L^\infty$, we have
\begin{eqnarray}\label{LA}
\sum\limits_{n=1}^\infty|\lambda_n D_n|^2&\lesssim&\sum\limits_{n=1}^\infty\left|\int\limits_0^1q(x)u_0(x)\phi_n(x)dx\right|^2+\sum\limits_{n=1}^\infty\left| \int\limits_0^1 u''_0(x)\phi_n(x)dx\right|^2\nonumber\\
&=&\sum\limits_{n=1}^\infty\left|\langle (q u_0),\phi_n\rangle\right|^2+\sum\limits_{n=1}^\infty |u_{0,n}''|^2=\|q u_0\|^2_{L^2}+\|u_0''\|^2_{L^2}\nonumber\\
&\leq& \|q\|^2_{L^\infty} \|u_0\|^2_{L^2}+\|u_0''\|^2_{L^2}.
\end{eqnarray}
Thus, 
$$\|\partial_t u(t,\cdot)\|^2_{L^2}\lesssim \|a\|^2_{L^\infty[0,T]}\left(\|u''_0\|^2_{L^2}+\|q\|^2_{L^\infty}\|u_0\|^2_{L^2}\right),$$
proving \eqref{ec2}.

Taking into account \eqref{u_x}, \eqref{e-val}, using \eqref{LA} and Parseval's identity we get
\begin{eqnarray*}
\|\partial_x u(t,\cdot)\|^2_{L^2}&\lesssim&\sum\limits_{n=1}^\infty\left|\sqrt{\lambda_n}D_n\right|^2\left(1+\|\nu\|^2_{L^2}\right)+\sum\limits_{n=1}^\infty \left|D_n\right|^2\|\nu\|^2_{L^\infty}\\
&\leq& \sum\limits_{n=1}^\infty\left|\lambda_n D_n\right|^2\left(1+\|\nu\|^2_{L^2}\right)+\sum\limits_{n=1}^\infty \left|D_n\right|^2\|\nu\|^2_{L^\infty}\\
&\lesssim& \left(\|u''_0\|^2_{L^2}+\|q\|^2_{L^\infty}\|u_0\|^2_{L^2}\right)\left(1 +\|\nu\|^2_{L^2}\right)+  \|u_0\|^2_{L^2}\|\nu\|^2_{L^\infty},
\end{eqnarray*}
implying \eqref{ec3}. 

Using \eqref{u_xx}, \eqref{LA} and Parseval’s identity  we obtain
\begin{eqnarray*}
\left\|\partial_x^2u(t, \cdot)\right\|^2_{L^2}&\lesssim&\|q\|^2_{L^\infty}\sum\limits_{n=1}^\infty |D_n|^2 +\sum\limits_{n=1}^\infty \left|\lambda_n D_n\right|^2\\
&\lesssim& \|q\|^2_{L^\infty}\|u_0\|^2_{L^2}+\|u''_0\|^2_{L^2}+\|q\|^2_{L^\infty}\|u_0\|^2_{L^2}=\|u''_0\|^2_{L^2}+2\|q\|^2_{L^\infty}\|u_0\|^2_{L^2}.
\end{eqnarray*}
The proof of Corollary \ref{cor1} is complete.
\end{proof}

Using the statements for the homogeneous case one can establish the following statements for the non-homogeneous Schrödinger initial/boundary problem \eqref{nonh}.

\begin{thm}\label{non-hom}
Assume that $q \in L^\infty(0,1)$, $q\geq0$, $a\in L^\infty [0,T]$, $a(t)\geq a_0>0$ for all $t\in[0,T]$, and $f\in C^1([0,T],W^k_\mathcal{L}(0,1))$ for some $k\in \mathbb{R}$. If the initial condition satisfies $u_0 \in W^{k}_\mathcal{L}(0,1)$ then the non-homogeneous Schrödinger equation with initial/boundary conditions \eqref{nonh} has the unique solution $u\in C([0,T], W^{k}_\mathcal{L})$ which satisfies the estimates
\begin{equation}\label{es-nh1}
        \|u(t,\cdot)\|_{L^2}\lesssim \|u_0\|_{L^2}+T\|f\|_{C([0,T],L^2(0,1))},
\end{equation}
\begin{eqnarray}\label{es-nh2}
\|\partial_tu(t,\cdot)\|_{L^2}&\lesssim& \|a\|_{L^\infty[0,T]}\left(\|u_0\|_{W^{2s}_\mathcal{L}}+T \|f\|_{C^1([0,T],W^{2s}_{\mathcal{L}}(0,1))}\right)\nonumber\\
&+&T \|f\|_{C^1([0,T],L^2(0,1))}.
\end{eqnarray}
When $s=1$, we also have
\begin{equation}\label{es-nh3}
\|\partial_xu(t,\cdot)\|_{L^2} \lesssim \left(1+\|\nu\|_{L^\infty}\right)\left(\|u_0\|_{W^1_{\mathcal{L}}}+T\|f\|_{C([0,T],W^1_{\mathcal{L}}(0,1)}\right),
\end{equation}
\begin{eqnarray}\label{es-nh4}
\|\partial^2_xu(t,\cdot)\|_{L^2}&\lesssim& \|q\|_{L^\infty} \left(\|u_0\|_{L^2}+T\|f\|_{C([0,T],L^2(0,1))}\right)\nonumber\\
&+&\|u_0\|_{W^2_\mathcal{L}}+T\|f\|_{C^1([0,T],{W^2_\mathcal{L}}(0,1))},
\end{eqnarray}
where the constants in these inequalities are independent of $u_0$, $q$, $a$ and $f$.
\end{thm}

\begin{proof}
We can use the eigenfunctions \eqref{sol-SL} of the corresponding (homogeneous) eigenvalue problem \eqref{4}, and look for a solution in the series form
\begin{equation}\label{nonhu}
    u(t,x)=\sum\limits_{n=1}^\infty u_n(t)\phi_n(x),
\end{equation}
where
$$u_n(t) =\int\limits_0^1 u(t,x)\phi_n(x)dx.$$
We can similarly expand the source function,
\begin{equation}\label{func}
    f(t,x)=\sum\limits_{n=1}^\infty f_n(t)\phi_n(x),\qquad f_n(t)=\int\limits_0^1f(t,x)\phi_n(x)dx.
\end{equation}

Now, since we are looking for a twice differentiable function $u(t,x)$ that satisfies the homogeneous
Dirichlet boundary conditions, we can use \eqref{l^s}  to the Fourier series \eqref{nonhu} term by term and using that the $\phi_n(x)$ satisfies the equation \eqref{4} to obtain
\begin{equation}\label{L^su}
    \mathcal{L}^su(t,x) = \sum\limits_{n=1}^\infty \mathcal{L}^s\left(u_n(t)\phi_n(x)\right)=\sum\limits_{n=1}^\infty u_n(t)\lambda^s_n\phi_n(x).
\end{equation}

We can also differentiate the series \eqref{func} with respect to $t$ to obtain
\begin{equation}\label{utt}
    u_{t}(t,x) = \sum\limits_{n=1}^\infty u'_n(t)\phi_n(x),
\end{equation}
since the Fourier coefficients of $u_{t}(t,x)$ are
$$\int\limits_0^1u_{t}(t,x)\phi_n(x)dx=\frac{\partial}{\partial t}\left[\int\limits_0^1u(t,x)\phi_n(x)dx\right]=u'_n(t).$$
Differentiation under the above integral is allowed since the resulting integrand is continuous.

Substituting \eqref{utt} and \eqref{L^su} into the equation, and using \eqref{func}, we have
\begin{eqnarray*}
    i\sum\limits_{n=1}^\infty u'_n(t)\phi_n(x)&+&a(t)\sum\limits_{n=1}^\infty u_n(t)\lambda^s_n\phi_n(x)=\sum\limits_{n=1}^\infty f_n(t)\phi_n(x).
\end{eqnarray*}
Due to the completeness,
$$iu'_n(t)+\lambda^s_na(t)u_n(t)=f_n(t), \qquad n=1,2,...,$$
which are ODEs for the coefficients $u_n(t)$ of the series \eqref{nonhu}. By the method of variation of constants we get 
\begin{eqnarray*}
u_n(t)&=&D_n e^{i \lambda^s_n \int\limits_0^t a(\tau)d\tau}+ e^{i \lambda^s_n \int\limits_0^t a(\tau)d\tau} \int\limits_0^t e^{-i\lambda^s_n\int\limits_0^s a(\tau)d\tau} f_n(s)ds,
\end{eqnarray*}
where
$$D_n=\int\limits_0^1u_0(x)\phi_n(x)dx.$$
Thus, we can write a solution of the equation \eqref{nonh} in the form
\begin{eqnarray}\label{sol-nh}
u(t,x)&=&\sum\limits_{n=0}^\infty D_n e^{i \lambda^s_n \int\limits_0^t a(\tau)d\tau}\phi_n(x)\nonumber\\
&+&\sum\limits_{n=0}^\infty  e^{i \lambda^s_n \int\limits_0^t a(\tau)d\tau} \int\limits_0^t e^{-i\lambda^s_n\int\limits_0^sa(\tau)d\tau}f_n(s)ds\phi_n(x).
\end{eqnarray}

Let us estimate $\|u(t,\cdot)\|^2_{L^2}$. For this we use the estimates
\begin{eqnarray}\label{est-nonh}
\int\limits_0^1|u(t,x)|^2dx&=&\int\limits_0^1\left|\sum\limits_{n=0}^\infty D_n e^{i \lambda^s_n \int\limits_0^t a(\tau)d\tau}\phi_n(x) \right.\nonumber\\
&+&\left.\sum\limits_{n=0}^\infty  e^{i \lambda^s_n \int\limits_0^t a(\tau)d\tau} \int\limits_0^t e^{-i\lambda^s_n\int\limits_0^sa(\tau)d\tau}f_n(s)ds\phi_n(x)\right|^2 dx\nonumber \\
&\lesssim& \int\limits_0^1\left|\sum\limits_{n=0}^\infty D_n e^{i \lambda^s_n \int\limits_0^t a(\tau)d\tau}\phi_n(x)\right|^2dx\nonumber\\
&+&\int\limits_0^1\left|\sum\limits_{n=0}^\infty  e^{i \lambda^s_n \int\limits_0^t a(\tau)d\tau} \int\limits_0^t e^{-i\lambda^s_n\int\limits_0^sa(\tau)d\tau}f_n(s)ds\phi_n(x)\right|^2 dx\nonumber\\
&=& I_1 + I_2.
\end{eqnarray}
For $I_1$ by using \eqref{25-0}-\eqref{25} for the homogeneous case we have that
$$I_1:=\int\limits_0^1\left|\sum\limits_{n=0}^\infty D_n e^{i \lambda^s_n \int\limits_0^t a(\tau)d\tau}\phi_n(x)\right|^2dx\lesssim \|u_0\|^2_{L^2}.$$
Now we estimate $I_2$ in \eqref{est-nonh} taking into account that $s\in [0,t]$
\begin{eqnarray*}
I_2&:=&\int\limits_0^1\left|\sum\limits_{n=0}^\infty  e^{i \lambda^s_n \int\limits_0^t a(\tau)d\tau} \int\limits_0^t e^{-i\lambda^s_n\int\limits_0^sa(\tau)d\tau}f_n(s)ds\phi_n(x)\right|^2 dx\\
&\leq& \int\limits_0^1\left|\sum\limits_{n=0}^\infty  e^{i \lambda^s_n \int\limits_0^t a(\tau)d\tau} e^{-i\lambda^s_n\int\limits_0^ta(\tau)d\tau} \int\limits_0^t f_n(s)ds\phi_n(x)\right|^2 dx\\
&=&\int\limits_0^1\left|\sum\limits_{n=0}^\infty \int\limits_0^t f_n(s)ds\phi_n(x)\right|^2 dx\lesssim \sum\limits_{n=1}^\infty\left[\int\limits_0^t|f_n(s)|ds\right]^2.
\end{eqnarray*}
Using Holder's inequality and taking into account that $t\in [0,T]$ we get
$$\left[\int\limits_0^t|f_n(s)|ds\right]^2\leq \left[\int\limits_0^T 1\cdot| f_n(t)|dt\right]^2\leq T\int\limits_0^T| f_n(t)|^2dt,$$
since $f_n(t)$ is the Fourier coefficient of the function $f(t,x)$, and by Parseval's identity we obtain
$$
\sum\limits_{n=1}^\infty T\int\limits_0^T|f_n(t)|^2dt= T\int\limits_0^T\sum\limits_{n=1}^\infty|f_n(t)|^2dt = T\int\limits_0^T\|f(t,\cdot)\|^2_{L^2}dt.$$
Since
$$\|f\|_{C([0,T],L^2(0,1))}=\max\limits_{0\leq t\leq T}\|f(t,\cdot)\|_{L^2},$$
we arrive at the inequality
$$T\int\limits_0^T\|f(t,\cdot)\|^2_{L^2}dt\leq T^2\|f\|^2_{C([0,T],L^2(0,1))}.$$
Thus,
\begin{eqnarray}\label{I2}
I_2&:=&\int\limits_0^1\left|\sum\limits_{n=0}^\infty  e^{i \lambda_n \int\limits_0^t a(\tau)d\tau} \int\limits_0^t e^{-i\lambda_n\int\limits_0^sa(\tau)d\tau}f_n(s)ds\phi_n(x)\right|^2 dx\nonumber\\
&\lesssim&  T^2\|f\|^2_{C([0,T],L^2(0,1))}.
\end{eqnarray}

We finally get
$$
\|u(t,\cdot)\|^2_{L^2}\lesssim \|u_0\|^2_{L^2}+T^2\|f\|^2_{C([0,T],L^2(0,1))},
$$
implying \eqref{es-nh1}.

Let us estimate $\|\partial_tu(t,\cdot)\|_{L^2}$, for this we calculate $\partial_tu(t,x)$ as follows
\begin{eqnarray*}
\partial_tu(t,x)&=&\sum\limits_{n=0}^\infty i \lambda^s_na(t) D_n e^{i \lambda^s_n \int\limits_0^t a(\tau)d\tau}\phi_n(x) \\
&+&\sum\limits_{n=0}^\infty i\lambda^s_na(t) e^{i \lambda^s_n \int\limits_0^t a(\tau)d\tau} \int\limits_0^t e^{-i\lambda^s_n \int\limits_0^sa(\tau)d\tau}f_n(s)ds\phi_n(x)\\
&+&\sum\limits_{n=0}^\infty f_n(t)\phi_n(x).
\end{eqnarray*}
Then
\begin{eqnarray}\label{part_t}
\|\partial_tu(t,\cdot)\|^2_{L^2}&=&\int\limits_0^1|\partial_tu(t,x)|^2dx\lesssim \int\limits_0^1\left|\sum\limits_{n=0}^\infty i \lambda^s_n a(t) D_n e^{i \lambda^s_n \int\limits_0^t a(\tau)d\tau}\phi_n(x)\right|^2dx \nonumber \\
&+&\int\limits_0^1\left|\sum\limits_{n=0}^\infty i \lambda^s_na(t) e^{i \lambda^s_n \int\limits_0^t a(\tau)d\tau} \int\limits_0^t e^{-i\lambda^s_n\int\limits_0^sa(\tau)d\tau}f_n(s)ds\phi_n(x)\right|^2 dx\nonumber \\
&+&\int\limits_0^1\left|\sum\limits_{n=0}^\infty f_n(t)\phi_n(x)\right|^2 dx = J_1 + J_2 +J_3.
\end{eqnarray}
Here, for $J_1$ by using the \eqref{t26} and \eqref{21-1} and taking into account \eqref{func} for the function $f(t,x)$ in $J_3$, we obtain
$$\|\partial_tu(t,\cdot)\|^2_{L^2}\lesssim \|a\|^2_{L^\infty[0,T]}\|u_0\|^2_{W^{2s}_\mathcal{L}}+J_2+\|f(t,\cdot)\|^2_{L^2}.$$
To estimate $J_2$ conducting evaluations as in \eqref{I2} and taking into account \eqref{func}, we obtain
\begin{eqnarray*}\label{J_2}
J_2&:=&\int\limits_0^1\left|\sum\limits_{n=0}^\infty i\lambda^s_na(t) e^{i \lambda^s_n \int\limits_0^t a(\tau)d\tau} \int\limits_0^t e^{-i\lambda^s_n\int\limits_0^sa(\tau)d\tau}f_n(s)ds\phi_n(x)\right|^2 dx\nonumber \\
&\lesssim& \int\limits_0^1\sum\limits_{n=0}^\infty \left|\lambda^s_na(t) \int\limits_0^t f_n(s)ds\phi_n(x)\right|^2 dx\leq \|a\|^2_{L^\infty[0,T]}\sum\limits_{n=0}^\infty \left| \int\limits_0^t \lambda^s_nf_n(s)ds\right|^2 \\
&\leq&\|a\|^2_{L^\infty[0,T]}\sum\limits_{n=0}^\infty \left|\int\limits_0^T\lambda^s_nf_n(t)dt\right|^2 \leq T\|a\|^2_{L^\infty[0,T]}\int\limits_0^T\sum\limits_{n=0}^\infty \left|\lambda^s_nf_n(t)\right|^2dt\\
&=&T\|a\|^2_{L^\infty[0,T]}\int\limits_0^T\sum\limits_{n=0}^\infty \left|\int\limits_0^1\lambda^s_nf(t,x)\phi_n(x)dx\right|^2dt=T\|a\|^2_{L^\infty[0,T]}\int\limits_0^T \left\|\mathcal{L}^{2s}f(t,\cdot)\right\|^2_{L^2}dt\\
&\leq& T^2\|a\|^2_{L^\infty[0,T]}\|f\|^2_{C\left([0,T],W^{2s}_\mathcal{L}(0,1)\right)}
\end{eqnarray*}
Therefore
$$\|\partial_tu(t,\cdot)\|^2_{L^2}\lesssim \|a\|^2_{L^\infty[0,T]}\left(\|u_0\|^2_{W^{2s}_\mathcal{L}}+T^2 \|f\|^2_{C([0,T],W^{2s}_{\mathcal{L}}(0,1))}\right) + \|f\|^2_{C([0,T], L^2(0,1)},$$
implying \eqref{es-nh2}.

Let $s=1$. Then we carry out next estimate as follows:
\begin{eqnarray*}
\|\partial_xu(t,\cdot)\|^2_{L^2}&=&\int\limits_0^1|\partial_xu(t,x)|^2dx\lesssim \int\limits_0^1\left|\sum\limits_{n=0}^\infty D_n e^{i \lambda_n \int\limits_0^t a(\tau)d\tau}\phi'_n(x)\right|^2 dx\\
&+&\int\limits_0^1\left|\sum\limits_{n=0}^\infty  e^{i \lambda_n \int\limits_0^t a(\tau)d\tau} \int\limits_0^t e^{-i\lambda_n\int\limits_0^sa(\tau)d\tau}f_n(s)ds\phi'_n(x)\right|^2dx= K_1+K_2.
\end{eqnarray*}
Using \eqref{est3} we get
\begin{eqnarray*}
K_1&:=& \int\limits_0^1\left|\sum\limits_{n=0}^\infty D_n e^{i \lambda_n \int\limits_0^t a(\tau)d\tau}\phi'_n(x)\right|^2 dx\lesssim \|u_0\|^2_{W^1_\mathcal{L}}\left(1+\|\nu\|^2_{L^2}\right)+\|u_0\|^2_{L^2}\|\nu\|^2_{L^\infty}.
\end{eqnarray*}
For $K_2$, using \eqref{I2}, \eqref{est3} and \eqref{phi-der}, \eqref{norm-phi} for $\phi'_n$, we obtain
\begin{eqnarray*}
K_2&:=& \int\limits_0^1\left|\sum\limits_{n=0}^\infty  e^{i \lambda_n \int\limits_0^t a(\tau)d\tau} \int\limits_0^t e^{-i\lambda_n\int\limits_0^sa(\tau)d\tau}f_n(s)ds\phi'_n(x)\right|^2dx\\
&\leq&\int\limits_0^1\left|\sum\limits_{n=1}^\infty \int\limits_0^t f_n(s)ds\left(\frac{\sqrt{\lambda_n}r_n(x)\cos \theta_n(x)}{\|\Tilde{\phi}_n\|_{L^2}}+\nu(x)\phi_n(x)\right)\right|^2dx\\
&\lesssim& \sum\limits_{n=0}^\infty \left|\int\limits_0^t\left|\sqrt{\lambda_n}f_n(s)\right|ds\right|^2\left(1+\|\nu\|^2_{L^2}\right) + \|\nu\|^2_{L^\infty}T^2\|f\|^2_{C([0,T],L^2(0,1)}.
\end{eqnarray*}
Taking into account \eqref{func} and \eqref{I2} we get
\begin{eqnarray*}
\sum\limits_{n=0}^\infty \left|\int\limits_0^t\left|\sqrt{\lambda_n}f_n(s)\right|ds\right|^2&\leq&  \sum\limits_{n=0}^\infty \left|\int\limits_0^T\left|\sqrt{\lambda_n}f_n(t)\right|dt\right|^2 \leq T\int\limits_0^T\sum\limits_{n=0}^\infty \left|\sqrt{\lambda_n}f_n(t)\right|^2dt\\
&=&T\int\limits_0^T\sum\limits_{n=0}^\infty \left|\int\limits_0^1\sqrt{\lambda_n}f(t,x)\phi_n(x)dx\right|^2dt\\
&=&T\int\limits_0^T \left\|\mathcal{L}^{\frac{1}{2}}f(t,\cdot)\right\|^2_{L^2}dt\leq T^2\|f\|^2_{C\left([0,T],W^1_\mathcal{L}(0,1)\right)},
\end{eqnarray*}
and we finally obtain
\begin{eqnarray*}
\|\partial_xu(t,\cdot)\|^2_{L^2}&\lesssim&\left(\|u_0\|^2_{W^1_\mathcal{L}}+T^2\|f\|^2_{C\left([0,T],W^1_{\mathcal{L}}(0,1)\right)} \right)\left(1+\|\nu\|^2_{L^2}\right)\\
&+&\left(\|u_0\|^2_{L^2}+T^2\|f\|^2_{C([0,T],L^2(0,1))}\right)\|\nu\|^2_{L^\infty}\\
&\lesssim& \left(1+\|\nu\|^2_{L^\infty}\right)\left(\|u_0\|^2_{W^1_{\mathcal{L}}}+T^2\|f\|^2_{C([0,T],W^1_{\mathcal{L}}(0,1)}\right),
\end{eqnarray*}
which gives \eqref{es-nh3}.

For the estimate $\|\partial^2_xu(t,\cdot)\|_{L^2}$ we use that $\phi_n''(x)=(q(x)- \lambda_n)\phi_n(x)$, to deduce
\begin{eqnarray*}
\|\partial^2_xu(t,\cdot)\|^2_{L^2}&=&\int\limits_0^1|\partial^2_xu(t,x)|^2dx\\
&\lesssim& \int\limits_0^1\left|\sum\limits_{n=0}^\infty D_n e^{i \lambda_n \int\limits_0^t a(\tau)d\tau}(q(x)- \lambda_n)\phi_n(x)\right|^2dx \\
&+&\int\limits_0^1\left|\sum\limits_{n=0}^\infty  e^{i \lambda_n \int\limits_0^t a(\tau)d\tau} \int\limits_0^t e^{-i\lambda_n\int\limits_0^sa(\tau)d\tau}f_n(s)ds(q(x)- \lambda_n)\phi_n(x)\right|^2dx,
\end{eqnarray*}
and using \eqref{est2.4}, \eqref{I2}, we arrive at the estimates
\begin{eqnarray*}
\|\partial^2_xu(t,\cdot)\|^2_{L^2}&\lesssim& \|q\|^2_{L^\infty} \left(\|u_0\|^2_{L^2}+T^2\|f\|^2_{C([0,T],L^2(0,1))}\right)\\
&+&\|u_0\|^2_{W^2_\mathcal{L}}+T^2\|f\|^2_{C^1([0,T],{W^2_\mathcal{L}}(0,1))}.
\end{eqnarray*}
The proof of Theorem \ref{non-hom} is complete.
\end{proof}

\begin{cor}\label{cor2} 
Let $s=1$. Assume that $q\in L^2(0,1)$, $q\geq0$, $a\in L^\infty [0,T]$, $a(t)\geq a_0>0$ for all $t\in[0,T]$, and $f\in C^1([0,T],L^2(0,1))$. If the initial condition satisfies $u_0 \in L^2(0,1)$ and $u_0''\in L^2(0,1)$, then the non-homogeneous Schrödinger equation with initial/boundary conditions  \eqref{nonh} has a unique solution $u\in C([0,T], L^2(0,1))$ such that
\begin{equation}\label{ec-nh1}
    \|u(t,\cdot)\|_{L^2}\lesssim \|u_0\|_{L^2}+T\|f\|_{C([0,1],L^2(0,1))},
\end{equation}
\begin{eqnarray}\label{ec-nh2}
\|\partial_t u(t,\cdot)\|_{L^2}&\lesssim&\|a\|_{L^\infty[0,T]}\left(\|u_0''\|_{L^2}+\|q\|_{L^\infty}\|u_0\|_{L^2}+\frac{T}{a_0} \|f\|_{C^1([0,T],L^2(0,1))}\right)\nonumber\\
&+&\|f\|_{C([0,T],L^2(0,1))},
\end{eqnarray}
\begin{eqnarray}\label{ec-nh3}
\|\partial_x u(t,\cdot)\|_{L^2}
&\lesssim& \left(\|u''_0\|_{L^2}+\|q\|_{L^\infty}\|u_0\|_{L^2}+\frac{T}{a_0}\|f\|_{C^1([0,T],L^2(0,1))}\right)\left(1 +\|\nu\|_{L^2}\right)\nonumber\\
&+&\|\nu\|_{L^\infty}\left(\|u_0\|_{L^2}+T\|f\|_{C([0,T],L^2(0,1))}\right),
\end{eqnarray}
\begin{eqnarray}\label{ec-nh4}
\left\|\partial_x^2u(t, \cdot)\right\|_{L^2}&\lesssim&\|u''_0\|_{L^2}+\frac{T}{a_0}\|f\|_{C^1([0,T],L^2(0,1))}\nonumber\\
&+&\|q\|_{L^\infty}\left(\|u_0\|_{L^2}+T|f\|_{C([0,T],L^2(0,1))}\right),
\end{eqnarray}
where the constants in these inequalities are independent of $u_0$, $q$, $a$ and $f$.
\end{cor}
\begin{proof}
    The inequality \eqref{ec-nh1} follows from Theorem \ref{non-hom}. For $\|\partial_tu(t,\cdot)\|_{L^2}$, using \eqref{part_t} we have
\begin{eqnarray*}
\|\partial_tu(t,\cdot)\|_{L^2}&\lesssim& \int\limits_0^1\left|\sum\limits_{n=0}^\infty i \lambda_n a(t) D_n e^{i \lambda_n \int\limits_0^t a(\tau)d\tau}\phi_n(x)\right|^2dx \nonumber \\
&+&\int\limits_0^1\left|\sum\limits_{n=0}^\infty i \lambda_na(t) e^{i \lambda_n \int\limits_0^t a(\tau)d\tau} \int\limits_0^t e^{-i\lambda_n\int\limits_0^sa(\tau)d\tau}f_n(s)ds\phi_n(x)\right|^2 dx\nonumber \\
&+&\int\limits_0^1\left|\sum\limits_{n=0}^\infty f_n(t)\phi_n(x)\right|^2 dx = J_1 + J_2 +J_3.
\end{eqnarray*}
According to \eqref{t26}, \eqref{21-1} and \eqref{LA} we get
\begin{eqnarray*}
J_1&:=&\int\limits_0^1\left|\sum\limits_{n=0}^\infty i \lambda_n a(t) D_n e^{i \lambda_n \int\limits_0^t a(\tau)d\tau}\phi_n(x)\right|^2dx\\
&\lesssim& \|a\|^2_{L^\infty[0,T]}\left(\|q\|^2_{L^\infty}\|u_0\|^2_{L^2}+\|u_0''\|^2_{L^2}\right).
\end{eqnarray*}
For $J_3$ taking into account \eqref{func} and Parseval's identity we obtain
\begin{eqnarray*}    J_3&:=&\int\limits_0^1\left|\sum\limits_{n=0}^\infty f_n(t)\phi_n(x)\right|^2 dx = \left|\sum\limits_{n=0}^\infty f_n(t)\right|^2\lesssim \sum\limits_{n=0}^\infty \left|f_n(t)\right|^2=\|f(t,\cdot)\|^2_{L^2}\\
&\leq& \|f\|_{C([0,T],L^2(0,1))}.
\end{eqnarray*}
To estimate $J_2$, integrating by parts we obtain 
\begin{eqnarray*}
J_2&=&\int\limits_0^1\left|\sum\limits_{n=0}^\infty i\lambda_n a(t)e^{i\lambda_n\int\limits_0^ta(\tau)d\tau} \int\limits_0^t e^{-i\lambda_n\int\limits_0^sa(\tau)d\tau}f_n(s)ds\phi_n(x)\right|^2 dx\nonumber \\
&=&\int\limits_0^1\left|\sum\limits_{n=0}^\infty \left(i\lambda_n a(t)e^{i\lambda_n\int\limits_0^ta(\tau)d\tau}\frac{i}{\lambda_n}e^{-i\lambda_n \int\limits_0^sa(\tau)d\tau}\frac{f_n(s)}{a(s)}\biggl|_0^t\right.\right. \nonumber\\
&+& \left.\left. a(t)e^{i\lambda_n\int\limits_0^ta(\tau)d\tau}\int\limits_0^t e^{-i\lambda_n\int\limits_0^sa(\tau)d\tau}\left(\frac{f_n(s)}{a(s)}\right)'ds\right)\phi_n(x)\right|^2 dx\\
&\lesssim& \sum\limits_{n=0}^\infty \left|a(t)\frac{f_n(s)}{a(s)}\biggl|_0^t\right|^2 + \sum\limits_{n=1}^\infty\left|a(t)\int\limits_0^t \left(\frac{f'_n(s)}{a(s)}-\frac{f_n(s)a'(s)}{a^2(s)}\right)ds\right|^2=J_{21}+J_{22}.
\end{eqnarray*}
Since $a(x)\geq a_0>0$ in $t\in[0,T]$, we can use the estimate $\left|\frac{1}{a(t)}\right|^2\leq \frac{1}{a_0^2}$ for all $t\in[0,T]$, and using Parseval's identity, we obtain
\begin{eqnarray*}
J_{21}&:=& \sum\limits_{n=0}^\infty \left|a(t) \frac{f_n(s)}{a(s)}\biggl|_0^t\right|^2 \lesssim \|a\|^2_{L^\infty[0,T]}\left(\sum\limits_{n=1}^\infty \left|\frac{f_n(t)}{a(t)}\right|^2+\sum\limits_{n=1}^\infty \left|\frac{f_n(0)}{a(0)}\right|^2\right)\\
&\leq& \frac{1}{ a_0^2}\|a\|^2_{L^\infty[0,T]}\left(\sum\limits_{n=1}^\infty \left|f_n(t)\right|^2+\sum\limits_{n=1}^\infty \left|f_n(0)\right|^2\right)\\
&=&\frac{1}{a_0^2} \|a\|^2_{L^\infty[0,T]}\left(\|f(t,\cdot)\|^2_{L^2}+\|f(0,\cdot)\|^2_{L^2}\right).
\end{eqnarray*}
Carrying out similar reasoning, integrating by parts and using \eqref{I2}, we get
\begin{eqnarray*} J_{22}&:=&\sum\limits_{n=1}^\infty\left|a(t)\int\limits_0^t \left(\frac{f'_n(s)}{a(s)}-\frac{f_n(s)a'(s)}{a^2(s)}\right)ds\right|^2\\
&=&\sum\limits_{n=0}^\infty\left|a(t) \left(\int\limits_0^t\frac{f'_n(s)}{a(s)}ds+\int\limits_0^tf_n(s)d\left(\frac{1}{a(s)}\right)\right)\right|^2\\
&\lesssim&\frac{1}{a_0^2}\|a\|^2_{L^\infty[0,T]}\left|\sum\limits_{n=1}^\infty\int\limits_0^t|f'_n(s)|ds\right|^2+\|a\|^2_{L^\infty[0,T]}\left|\sum\limits_{n=1}^\infty \left(\frac{f_n(s)}{a(s)}\biggl|_0^t-\int\limits_0^t\frac{f'(s)}{a(s)}ds\right)\right|^2\\
&\lesssim&\frac{T^2}{a_0^2}\|a\|^2_{L^\infty[0,T]}\|f'\|^2_{C([0,T],L^2(0,1))}+\frac{1}{a_0^2}\|a\|^2_{L^\infty[0,T]}\left(\|f(t,\cdot)\|^2_{L^2}+\|f(0,\cdot)\|^2_{L^2}\right)\\
&+&\frac{T^2}{a_0^2}\|a\|^2_{L^\infty[0,T]}\|f'\|^2_{C([0,T],L^2(0,1))}.
\end{eqnarray*}
And finally, for $J_2$ we have
\begin{eqnarray}\label{f'}
J_2&:=&\int\limits_0^1\left|\sum\limits_{n=0}^\infty \lambda_n e^{-\lambda_n\int\limits_0^ta(\tau)d\tau} \int\limits_0^t e^{\lambda_n\int\limits_0^sa(\tau)d\tau}f_n(s)ds\phi_n(x)\right|^2 dx\nonumber\\
&\lesssim& \frac{2}{a_0^2}\|a\|^2_{L^\infty[0,T]} \left(\|f(t,\cdot)\|^2_{L^2}+\|f(0,\cdot)\|^2_{L^2}\right) +2 \frac{T^2}{a_0^2}\|a\|^2_{L^\infty[0,T]}\|f'\|^2_{C([0,T],L^2(0,1)}\nonumber\\
&\leq& \frac{2T^2}{a_0^2}\|a\|^2_{L^\infty[0,T]} \|f\|^2_{C^1([0,T],L^2(0,1))}.
\end{eqnarray}
Therefore,
\begin{eqnarray*}
\|\partial_tu(t,\cdot)\|^2_{L^2}&\lesssim& \|a\|^2_{L^\infty[0,T]}\left(\|u_0''\|^2_{L^2}+\|q\|^2_{L^\infty}\|u_0\|^2_{L^2}+\frac{2T^2}{a_0^2} \|f\|^2_{C^1([0,T],L^2(0,1))}\right)\\
&+&\|f\|^2_{C([0,T],L^2(0,1))},    
\end{eqnarray*}
implying \eqref{ec-nh2}.
Taking into account Corollary \ref{cor1} and similarly to previous estimates, we obtain the inequalities \eqref{ec-nh3} and \eqref{ec-nh4}. 

The proof of Corollary \ref{cor2} is complete.
\end{proof}

\section{Very weak solutions}

In this section we consider the differential case $s=1$. We will analyse the solutions for less regular coefficients $q$, $a$ and the initial condition $u_0$. For this we will be using the notion of very weak solutions. 

Assume that the coefficient $q$ and initial condition $u_0$ are the distributions on $(0,1)$, the coefficient $a$ is distribution on $[0,T]$. To regularize distributions, we introduce the following definition.

\begin{defn}\label{D1}  (i)
A net of functions $\left(u_\varepsilon=u_\varepsilon(t,x)\right)$ is said to be uniformly $L^2$-moderate if there exist $N\in \mathbb{N}_0$ and $C>0$ such that 
$$\|u_\varepsilon(t,\cdot)\|_{L^2}\leq C \varepsilon^{-N}, \quad \text{for all } t\in[0,T].$$
(ii) A net of functions ($u_{0,\varepsilon}=u_{0,\varepsilon}(x)$) is said to be $H^2$-moderate if there exist $N\in \mathbb{N}_0$ and $C>0$ such that
$$\|u_{0,\varepsilon}\|_{L^2}\leq C\varepsilon^{-N},\qquad \|u''_{0,\varepsilon}\|_{L^2}\leq C\varepsilon^{-N}.$$
\end{defn}

\begin{defn}\label{D11} (i) A net of functions $\left(q_\varepsilon=q_\varepsilon(x)\right)$ is said to be $L^\infty$-moderate if there exist $N\in \mathbb{N}_0$ and $C>0$ such that 
$$\|q_\varepsilon\|_{L^\infty(0,1)}\leq C \varepsilon^{-N}.$$
(ii) A net of functions $(a_\varepsilon=a_\varepsilon
(t))$ is said to be $L^\infty$-moderate if there exist $N\in \mathbb{N}_0$ and $C>0$ such that
$$\|a_\varepsilon\|_{L^\infty[0,T]}\leq C \varepsilon^{-N}.$$

\end{defn}

\begin{rem} We note that such assumptions are natural for distributional coefficients in the sense that regularisations of distributions are moderate. Precisely, by the structure theorems for distributions (see, e.g. \cite{Friedlander}), we know that distributions 
\begin{equation}\label{moder}
  \mathcal{D}'(0,1) \subset \{L^\infty(0,1) -\text{moderate families} \},  
\end{equation}
and we see from \eqref{moder}, that a solution to an initial/boundary problem may not exist in the sense of distributions, while it may exist in the set of $L^\infty$-moderate functions. 
\end{rem}

To give an example, let us take $f\in L^2(0,1)$, $f:(0,1)\to \mathbb{C}$. We introduce the function 
$$\Tilde{f}=\left\{\begin{array}{l}
    f, \text{ on }(0,1),  \\
    0,  \text{ on }\mathbb{R} \setminus (0,1),
\end{array}\right.$$
then $\Tilde{f}:\mathbb{R}\to \mathbb{C}$, and $\Tilde{f}\in \mathcal{E}'(\mathbb{R}).$

Let $\Tilde{f}_\varepsilon=\Tilde{f}*\psi_\varepsilon$ be obtained as the convolution of $\Tilde{f}$ with a Friedrich mollifier $\psi_\varepsilon$, where 
$$\psi_\varepsilon(x)=\frac{1}{\varepsilon}\psi\left(\frac{x}{\varepsilon}\right),\quad \text{for}\,\, \psi\in C^\infty_0(\mathbb{R}),\, \int \psi=1. $$
Then the regularising net $(\Tilde{f}_\varepsilon)$ is $L^p$-moderate for any $p \in [1,\infty)$, and it approximates $f$ on $(0,1)$:
$$0\leftarrow \|\Tilde{f}_\varepsilon-\Tilde{f}\|^p_{L^p(\mathbb{R})}\approx \|\Tilde{f}_\varepsilon-f\|^p_{L^p(0,1)}+\|\Tilde{f}_\varepsilon\|^p_{L^p(\mathbb{R}\setminus (0,1))}.$$
Now, let us introduce the notion of a very weak solution to the initial/boundary problem \eqref{C.p1}-\eqref{C.p3}.

\begin{defn}\label{D2}
Let $q\in \mathcal{D}'(0,1)$, $a\in \mathcal{D}'[0,T]$.  The net $(u_\varepsilon)_{\varepsilon>0}$ is said to be a very weak solution to the initial/boundary problem  \eqref{C.p1}-\eqref{C.p3} if there exists an $L^\infty$-moderate regularisation $q_\varepsilon$ of $q$, $L^\infty$-moderate regularisation $a_\varepsilon$ of $a$, and an $H^2$-moderate regularisation $u_{0,\varepsilon}$ of $u_0$, such that
\begin{equation}\label{vw1}
         \left\{\begin{array}{l}i\partial_t u_\varepsilon(t,x)+a_\varepsilon(t)\left(-\partial^2_x u_\varepsilon(t,x)+q_\varepsilon(x) u_\varepsilon(t,x)\right)=0,\quad (t,x)\in [0,T]\times(0,1),\\
 u_\varepsilon(0,x)=u_{0,\varepsilon}(x),\,\,\, x\in (0,1), \\
u_\varepsilon(t,0)=0=u_\varepsilon(t,1), \quad t\in[0,T],
\end{array}\right.\end{equation}
and $(u_\varepsilon)$ and $(\partial_t u_\varepsilon)$ are uniformly $L^{2}$-moderate.
\end{defn}

Then we have the following properties of very weak solutions.

\begin{thm}[Existence]\label{Ext}
Let the coefficients $q$ and initial condition $u_0$ be distributions in $(0,1)$, $q\geq0$, the coefficient $a$ be distribution in $[0,T]$ and there exists $a_0>0$, such that  $a\geq a_0>0$ in the sense that $\langle a-a_0,\phi\rangle \geq0$ for any $\phi\geq0$. Then the initial/boundary problem  \eqref{C.p1}-\eqref{C.p3} has a very weak solution.
\end{thm}
\begin{proof}
Since the formulation of \eqref{C.p1}-\eqref{C.p3} in this case might be impossible in the distributional sense due to issues related to the product of distributions, we replace \eqref{C.p1}-\eqref{C.p3} with a regularised equation. In other words, we regularize $q$ and $u_0$ by some corresponding sets $q_\varepsilon\geq 0$ and $u_{0,\varepsilon}$ of smooth functions from $ C^ \infty(0,1)$, $a$ by the set $a_\varepsilon$ of smooth functions from $C^\infty[0,T]$.

Hence, $q_\varepsilon$, $a_\varepsilon$ are $L^\infty$-moderate regularisations and $u_{0,\varepsilon}$ is an $H^2$-moderate regularisation of the coefficients $q$, $a$ and the Cauchy condition $u_0$ respectively. So by Definition \ref{D1} there exist $N\in \mathbb{N}_0$ and $C_1>0,\,C_2>0$, $C_3$, $C_4$ such that
$$\|q_\varepsilon\|_{L^\infty}\leq C_1\varepsilon^{-N},\quad \|u_{0,\varepsilon}\|_{L^2}\leq C_2\varepsilon^{-N},\quad  \|u''_{0,\varepsilon}\|_{L^2}\leq C_3 \varepsilon^{-N}, \quad \|a\|_{L^\infty}\leq C_4\varepsilon^{-N}.$$

Now we fix $\varepsilon\in (0,1]$, and consider the regularised problem \eqref{vw1}. Then all discussions and calculations of Theorem \ref{th1} are valid. Thus, by Corollary \ref{cor1}, the equation \eqref{vw1} has a unique solution $u_\varepsilon(t,x)$ in the space $C([0,T];L^2(0,1))$.

By Corollary \ref{cor1}, there exist $N\in \mathbb{N}_0$ and $C>0$, such that
$$\|u_\varepsilon(t,\cdot)\|_{L^2}\lesssim \|u_{0,\varepsilon}\|_{L^2}\leq C\varepsilon^{-N},$$
$$\|\partial_t u_\varepsilon(t,\cdot)\|_{L^2}\lesssim \|a\|^2_{L^\infty[0,T]}\left(\|u''_{0,\varepsilon}\|_{L^2}+\|q_\varepsilon\|_{L^\infty}\|u_{0, \varepsilon}\|_{L^2}\right)\leq C\varepsilon^{-N},$$
where the constants in these inequalities are independent of $u_0$, $q$ and $a$.
Therefore, $(u_\varepsilon)$ is uniformly $L^2$-moderate, and the proof of Theorem \ref{Ext} is complete.
\end{proof}

Describing the uniqueness of the very weak solutions amounts to “measuring” the changes of involved associated nets: negligibility conditions for nets of functions/distributions read as follows:

\begin{defn}[Negligibility]\label{D3}
Let $(u_\varepsilon)$, $(\Tilde{u}_\varepsilon)$ be two nets in $L^2(0,1)$. Then, the net $(u_\varepsilon-\Tilde{u}_\varepsilon)$ is called $L^2$-negligible, if for every $N\in \mathbb{N}$ there exist $C>0$ such that the following condition is satisfied
$$\|u_\varepsilon-\Tilde{u}_\varepsilon\|_{L^2}\leq C \varepsilon^N,$$
for all $\varepsilon\in (0,1]$. In the case where $u_\varepsilon=u_\varepsilon(t,x)$ is a net depending on $t\in [0,T]$, then the uniformly $L^2$-negligibility condition can be described as 
$$\|u_\varepsilon(t,\cdot)-\Tilde{u}_\varepsilon(t,\cdot)\|_{L^2}\leq C \varepsilon^N,$$
uniformly in $t\in [0,T]$. The constant $C$ can depends on $N$ but not on $\varepsilon$.
\end{defn}

Let us state the  \textquotedblleft$\varepsilon$-parameterised problems\textquotedblright \,to be considered:
\begin{equation}\label{un1}
    \left\{\begin{array}{l}i\partial_t u_\varepsilon(t,x)+a_\varepsilon(t)\left(-\partial^2_x u_\varepsilon(t,x)+q_\varepsilon(x) u_\varepsilon(t,x)\right)=0,\quad (t,x)\in [0,T]\times(0,1),\\
 u_\varepsilon(0,x)=u_{0,\varepsilon}(x),\,\,\, x\in (0,1), \\
u_\varepsilon(t,0)=0=u_\varepsilon(t,1),\quad t\in[0,T],
\end{array}\right.
\end{equation}
and
\begin{equation}\label{un2}
    \left\{\begin{array}{l}i\partial_t \Tilde{u}_\varepsilon(t,x)+\Tilde{a}_{\varepsilon}(t)\left(-\partial^2_x \Tilde{u}_\varepsilon(t,x)+\Tilde{q}_\varepsilon(x) \Tilde{u}_\varepsilon(t,x)\right)=0,\quad (t,x)\in [0,T]\times(0,1),\\
 \Tilde{u}_\varepsilon(0,x)=\Tilde{u}_{0,\varepsilon}(x),\,\,\, x\in (0,1), \\
\Tilde{u}_\varepsilon(t,0)=0=\Tilde{u}_\varepsilon(t,1), \quad t\in[0,T].
\end{array}\right.
\end{equation}

\begin{defn}[Uniqueness of the very weak solution]\label{D4}
Let $q\in \mathcal{D}'(0,1)$, $a\in \mathcal{D}'[0,T]$. We say that initial/boundary problem \eqref{C.p1}-\eqref{C.p3} has an unique very weak solution, if
for all $L^\infty$-moderate nets $q_\varepsilon$, $\Tilde{q}_\varepsilon$, such that $(q_\varepsilon-\Tilde{q}_\varepsilon)$ is $L^\infty$-negligible; for all $L^\infty$-moderate nets $a_\varepsilon$, $\Tilde{a}_\varepsilon$, such that $(a_\varepsilon-\Tilde{a}_\varepsilon)$ is $L^\infty$-negligible; and for all $H^2$-moderate regularisations $u_{0,\varepsilon},\,\Tilde{u}_{0,\varepsilon}$, such that $(u_{0,\varepsilon}-\Tilde{u}_{0,\varepsilon})$, is $H^2$-negligible, we have that $u_\varepsilon-\Tilde{u}_\varepsilon$ is uniformly $L^2$-negligible.
\end{defn}

\begin{thm}[Uniqueness of the very weak solution]\label{Th-U}
Let the coefficient $q$ and initial condition $u_0$ be distributions in $(0,1)$, $q\geq0$, the coefficient $a$ be a distribution in $[0,T]$ and there exists $a_0>0$, such that $a\geq a_0>0$ in the sense that $\langle a-a_0,\phi\rangle\geq0$ for any $\phi\geq0$. Then the very weak solution to the initial/boundary problem  \eqref{C.p1}-\eqref{C.p3} is unique.
\end{thm}
\begin{proof}
We denote by $u_\varepsilon$ and $\Tilde{u}_\varepsilon$ the families of solutions to the initial/boundary problems \eqref{un1} and \eqref{un2}, respectively. Setting $U_\varepsilon$ to be the difference of these nets $U_\varepsilon:=u_\varepsilon(t,\cdot)-\Tilde{u}_\varepsilon(t,\cdot)$, then $U_\varepsilon$ solves
    \begin{equation}\label{unq}
    \left\{\begin{array}{l}i\partial_t U_\varepsilon(t,x)+a_\varepsilon(t)\left(-\partial^2_x U_\varepsilon(t,x)+q_\varepsilon(x) U_\varepsilon(t,x)\right)=f_\varepsilon(t,x),\quad (t,x)\in [0,T]\times(0,1),\\
 U_\varepsilon(0,x)=(u_{0,\varepsilon}-\Tilde{u}_{0,\varepsilon})(x),\,\,\, x\in (0,1), \\
U_\varepsilon(t,0)=0=U_\varepsilon(t,1),
\end{array}\right.
\end{equation}
where we set 
\begin{eqnarray*}
    f_\varepsilon(t,x)&:=&(a_\varepsilon(t)-\Tilde{a}_\varepsilon(t))\partial^2_x\Tilde{u}_\varepsilon(t,x)\\
    &+&\left\{\Tilde{a}_\varepsilon(t)\left(\Tilde{q}_\varepsilon(x)-q_\varepsilon(x)\right)+q_\varepsilon(x)\left(\Tilde{a}_\varepsilon(x)-a_\varepsilon(x)\right)\right\}\Tilde{u}_\varepsilon(t,x)
\end{eqnarray*} 
for the forcing term to the non-homogeneous initial/boundary problem \eqref{unq}.

Passing to the $L^2$-norm of the $U_\varepsilon$, by using \eqref{ec-nh1} we obtain 
$$
\|U_\varepsilon(t,\cdot)\|^2_{L^2}\lesssim \|U_\varepsilon(0,\cdot)\|^2_{L^2}+T^2\|f_\varepsilon\|^2_{C([0,T],L^2(0,1))}.
$$
For the $\|f_\varepsilon\|^2_{C([0,T],L^2(0,1))}$ by using \eqref{25}, \eqref{ec4} we get
\begin{eqnarray*}
\|f_\varepsilon\|^2_{C([0,T],L^2(0,1))}&\lesssim& \|\Tilde{a}_\varepsilon-a_\varepsilon\|^2_{L^\infty[0,T]}\left(\|\Tilde{u}''_{0,\varepsilon}\|^2_{L^2}+2\|\Tilde{q}_\varepsilon\|^2_{L^\infty}\|\Tilde{u}_{0,\varepsilon}\|^2_{L^2}\right)\\
&+&\|\Tilde{q}_\varepsilon-q_\varepsilon\|^2_{L^\infty}\|\Tilde{a}_\varepsilon\|^2_{L^\infty[0,T]}\|\Tilde{u}_\varepsilon\|^2_{C([0,T],L^2(0,1))}\\
&+&\|\Tilde{a}_\varepsilon-a_\varepsilon\|^2_{L^\infty[0,T]}\|q_\varepsilon\|^2_{L^\infty}\|\Tilde{u}_\varepsilon\|^2_{C([0,T],L^2(0,1))}.
\end{eqnarray*}
Next, using the initial condition of \eqref{unq}, we obtain
\begin{eqnarray*}
\|U_\varepsilon(t,\cdot)\|^2_{L^2}&\lesssim& \|u_{0,\varepsilon}-\Tilde{u}_{0,\varepsilon}\|^2_{L^2}+T^2\|\Tilde{a}_\varepsilon-a_\varepsilon\|^2_{L^\infty[0,T]}\left(\|\Tilde{u}''_{0,\varepsilon}\|^2_{L^2}+2\|\Tilde{q}_\varepsilon\|^2_{L^\infty}\|\Tilde{u}_{0,\varepsilon}\|^2_{L^2}\right)\\
&+&T^2\|\Tilde{q}_\varepsilon-q_\varepsilon\|^2_{L^\infty}\|\Tilde{a}_\varepsilon\|^2_{L^\infty[0,T]}\|\Tilde{u}_\varepsilon\|^2_{C([0,T],L^2(0,1))}\\
&+&T^2\|\Tilde{a}_\varepsilon-a_\varepsilon\|^2_{L^\infty[0,T]}\|q_\varepsilon\|^2_{L^\infty}\|\Tilde{u}_\varepsilon\|^2_{C([0,T],L^2(0,1))}.
\end{eqnarray*}
Taking into account the negligibility of the nets $u_{0,\varepsilon}-\Tilde{u}_{0,\varepsilon}$, $q_\varepsilon-\Tilde{q}_\varepsilon$ and $a_\varepsilon-\Tilde{a}_\varepsilon$ we get
$$\|U_\varepsilon(t,\cdot)\|^2_{L^2}\leq C_1\varepsilon^{N_1}+\varepsilon^{N_2}\left(C_2\varepsilon^{-N_3}+C_3\varepsilon^{-N_4}\right)+\varepsilon^{N_5}\left(C_4\varepsilon^{-N_6}+C_5\varepsilon^{-N_7}\right)$$
for some $C_1>0,\,C_2>0, \, C_3>0,\,C_4>0,\, C_5>0,\,\,N_3,\,N_4,\, N_6,\, N_7\in \mathbb{N}$ and all $N_1,\,N_2,\, N_5 \in \mathbb{N}$, since $\Tilde{u}_\varepsilon$ is moderate. Then, for some $C_M>0$ and all $M\in \mathbb{N}$
$$\|U_\varepsilon(t,\cdot)\|_{L^2}\leq C_M \varepsilon^M.$$
The last estimate holds true uniformly in $t$ , and this completes the proof of Theorem \ref{Th-U}.
\end{proof}

\begin{thm}[Consistency]\label{Th-C} Assume that $q\in L^\infty(0,1)$, $q\geq0$, $a(t)\geq a_0>0$ for all $t\in[0,T]$ and let $(q_\varepsilon)$ be any $L^\infty$-regularisation of $q$, $(a_\varepsilon)$ be any $L^\infty$-regularisation of $a$, that is $\|q_\varepsilon-q\|_{L^\infty}\to 0$, $\|a_\varepsilon-a\|_{L^\infty[0,T]}\to 0$ as $\varepsilon\to 0$. Let the initial condition satisfy $u_0 \in L^2(0,1)$. Let $u$ be a very weak solution of the initial/boundary problem \eqref{C.p1}-\eqref{C.p3}. Then for any families $q_\varepsilon$, $a_\varepsilon$ $u_{0,\varepsilon}$ such that  $\|q-q_{\varepsilon}\|_{L^\infty}\to 0$, $\|a-a_{\varepsilon}\|_{L^\infty[0,T]}\to 0$, $\|u_{0}-u_{0,\varepsilon}\|_{L^2}\to 0$ as $\varepsilon\to 0$, any representative $(u_\varepsilon)$ of the very weak solution converges as 
$$\sup\limits_{0\leq t\leq T}\|u(t,\cdot)-u_\varepsilon(t,\cdot)\|_{L^2}\to 0$$ 
for $\varepsilon\to 0$ to the unique classical solution $u\in C([0,T];L^2(0,1))$ of the initial/boundary problem \eqref{C.p1}-\eqref{C.p3} given by Theorem \ref{th1}.
\end{thm}

\begin{proof}
For $u$ and for $u_\varepsilon$, as in our assumption, we introduce an auxiliary notation $V_\varepsilon(t, x):= u(t,x)-u_\varepsilon(t,x)$. Then the net $V_\varepsilon$ is a solution to the initial/boundary problem
\begin{equation}\label{51}
    \left\{\begin{array}{l}
        i\partial_tV_\varepsilon(t,x)+a_\varepsilon(t)\left(-\partial^2_xV_\varepsilon(t,x)+q_\varepsilon(x)V_\varepsilon(t,x)\right)=f_\varepsilon(t,x),\\
        V_\varepsilon(0,x)=(u_0-u_{0,\varepsilon})(x),\quad x\in (0,1),\\        V_\varepsilon(t,0)=0=V_\varepsilon(t,1), \quad t\in[0,T],
    \end{array}\right.
\end{equation}
where 
$$f_\varepsilon(t,x):=(a(t)-a_\varepsilon(t))\partial^2_x u(t,x)+\left\{a_\varepsilon(t)(q_\varepsilon(x)-q(x))+q(x)(a_\varepsilon(t)-a(t))\right\}u(t,x).$$  Analogously to Theorem \ref{Th-U} we have that
\begin{eqnarray*}
    \|V_\varepsilon(t,\cdot)\|^2_{L^2}&\lesssim& \|u_{0}-u_{0,\varepsilon}\|^2_{L^2}+T^2\|a-a_\varepsilon\|^2_{L^\infty[0,T]}\left(\|u''\|^2_{L^2}+2\|q\|^2_{L^\infty}\|u\|^2_{L^2}\right)\\
&+&T^2\|a-a_\varepsilon\|^2_{L^\infty[0,T]}\|q\|^2_{L^\infty}\|u\|^2_{C([0,T],L^2(0,1))}\\
&+&T^2\|q-q_\varepsilon\|^2_{L^\infty}\|a_\varepsilon\|^2_{L^\infty[0,T]}\|u\|^2_{C([0,T],L^2(0,1))}.
\end{eqnarray*}
$$$$
Since
$$\|u_{0}-{u}_{0,\varepsilon}\|_{L^2}\to 0,\quad  \|q_\varepsilon-q\|_{L^\infty}\to 0, \quad \|a-a_\varepsilon\|_{L^\infty[0,T]}\to 0$$
for $\varepsilon\to 0$ and $u$ is a very weak solution of the initial/boundary problem \eqref{C.p1}-\eqref{C.p3} we get
$$\|V_\varepsilon(t,\cdot)\|_{L^2}\to 0$$
for $\varepsilon\to 0$. This proves Theorem \ref{Th-C}.
\end{proof}

\end{document}